\documentclass[12pt]{amsart}%
\usepackage[latin9]{inputenc}
\usepackage{amstext}
\usepackage{amsthm}
\usepackage{amssymb}
\usepackage{amsfonts}
\usepackage{amscd}
\usepackage{amsmath}
\usepackage{graphicx}%
\providecommand{\U}[1]{\protect\rule{.1in}{.1in}}
\providecommand{\U}[1]{\protect\rule{.1in}{.1in}}
\newtheorem{theorem}{Theorem}[section]
\newtheorem{lemma}{Lemma}[section]

\newtheorem{corollary}{Corollary}[section]

\numberwithin{equation}{section}

\theoremstyle{remark}
\newtheorem{remark}{Remark}[section]
\numberwithin{equation}{section}
\begin{document}
\title{A note on the Almost Schur lemma on smooth metric measure spaces}
\author{Jui-Tang Chen$^{\ast}$}
\address{Department of Mathematics, National Taiwan Normal University, Taipei 11677,
Taiwan, R.O.C.}
\email{jtchen@ntnu.edu.tw}
\thanks{$^{\ast}$Research supported in part by MOST of Taiwan.}

\begin{abstract}
In this paper, we prove almost Schur Lemma on closed smooth metric measure
spaces, which implies the results of X. Cheng \cite{C1} and De Lellis-Topping
\cite{LT} whenever the weighted function $f$ is constant.

\end{abstract}
\maketitle

\section{Introduction}

\bigskip{}

In 2012, De Lellis and C. Topping \cite{LT} proved an almost Schur Lemma, that
is, if a closed Riemannian manifold has nonnegative Ricci curvature, they
showed an almost Schur inequality involves scalar curvature and Ricci
curvature:
\[%
\begin{array}
[c]{c}%
\int_{M}\left(  R-\overline{R}\right)  ^{2}dv\leq\frac{4n\left(  n-1\right)
}{\left(  n-2\right)  ^{2}}\int_{M}\left\vert Ric-\frac{R}{n}g\right\vert
^{2}dv.
\end{array}
\]
In particular, the equality holds if and only if this manifold is Einstein and
constant scalar curvature.

Later, in \cite{GW}, Y. Ge and G. Wang proved the almost Schur Lemma without
the condition of nonnegative Ricci curvature in $4$-dimension closed
Riemannian manifold, i.e. they just assume the nonnegative of scalar curvature
(or see \cite{B} for $4$-dimension closed Riemannian manifold with Yamabe invariant).

In \cite{C1}, X. Cheng considered the closed Riemannian manifolds with
negative Ricci curvature and obtained a generalization of the De
Lellis-Topping's theorem (or see \cite{C2}). For more references, see
\cite{CDW}\cite{CSW}\cite{H}\cite{W}.

In this paper, we study almost Schur Lemma on a smooth metric measure space.
First of all, we recall some definitions of smooth metric measure space.

For an $n$-dimensional closed Riemannian manifolds $\left(  M^{n},g\right)  $
and a smooth function $f$ on $M.$ A triple $\left(  M^{n},g,dv_{f}\right)  $
is a smooth metric measure space with weighted volume identity $dv_{f}%
=e^{-f\left(  x\right)  }dv,$ where $dv$ is the volume element of $M$ with
respect to the metric $g.$ Let $\left(  \nabla f\otimes\nabla u\right)
_{ij}=\frac{1}{2}\left(  f_{,i}u_{,j}+f_{,j}u_{,i}\right)  ,$ and let $Hess$
be the Hessian of the metric $g,$ we define the weighted Laplacian by the
trace of
\[%
\begin{array}
[c]{l}%
\left(  Hess_{f}u\right)  _{ij}\equiv\left(  Hessu\right)  _{ij}-\left(
\nabla f\otimes\nabla u\right)  _{ij},
\end{array}
\]
i.e.
\[%
\begin{array}
[c]{l}%
\Delta_{f}u=\Delta u-\left\langle \nabla f,\nabla u\right\rangle
\end{array}
\]
which is a self-adjoint operator concerning $dv_{f}.$

It is natural to consider the $m$-Bakry-Émery and $\infty$-Bakry-Émery Ricci
tensor on smooth metric measure space by
\[%
\begin{array}
[c]{l}%
Ric_{f}^{m}=Ric+Hessf-\frac{1}{m}\nabla f\otimes\nabla f,\text{ }m>0,
\end{array}
\]
and
\[%
\begin{array}
[c]{l}%
Ric_{f}=Ric+Hessf
\end{array}
\]
respectively. If $Ric_{f}=\lambda g$ (or $Ric_{f}^{m}=\lambda g$) for some
$\lambda\in%
\mathbb{R}
$, then $M$ is quasi-Einstein (or $m$-quasi-Einstein), and $Ric_{f}=\lambda g
$ is the gradient Ricci soliton equation. Whenever $f$ is a constant, then $M
$ is called trivial Einstein.

According to the classical Bochner's formula, we have a similar formula
\[%
\begin{array}
[c]{l}%
\frac{1}{2}\Delta_{f}\left\vert \nabla u\right\vert ^{2}=\left\vert
Hessu\right\vert ^{2}+\left\langle \nabla u,\nabla\Delta_{f}u\right\rangle
+Ric_{f}\left(  \nabla u,\nabla u\right)
\end{array}
\]
for $u\in C^{3}\left(  M\right)  $ on $M.$ Hence, there are many results may
be extended from Riemannian manifolds to smooth metric measure spaces. We
refer the reader to, for example \cite{BQ}\cite{CZ}\cite{MW1}\cite{MW2}%
\cite{W}\cite{WE}\cite{WCS}\cite{WW} for further references.

In fact, if $R$ is the scalar curvature of $M$ with respect to the metric $g,
$
\[%
\begin{array}
[c]{l}%
N_{f}^{m}\equiv\left(  R+\frac{2\left(  m-1\right)  }{m}\Delta f-\frac{m-1}%
{m}\left\vert \nabla f\right\vert ^{2}\right)  e^{-\frac{2}{m}}%
\end{array}
\]
and
\[%
\begin{array}
[c]{l}%
\overline{N}_{f}^{m}=\frac{\int_{M}N_{f}^{m}dv_{f}}{\int_{M}dv_{f}},
\end{array}
\]
J.-Y. Wu \cite{W} generalized De Lellis-Topping's result as follows: Let
$\left(  M^{n},g,e^{-f}dv\right)  $ be a closed smooth metric measure space.
For any positive number $m\neq2,$ if
\[%
\begin{array}
[c]{l}%
Ric_{f}^{m}\geq\frac{\left\vert \nabla f\right\vert ^{2}}{m}g,
\end{array}
\]
then
\[%
\begin{array}
[c]{l}%
\int_{M}\left(  N_{f}^{m}-\overline{N}_{f}^{m}\right)  ^{2}e^{-f}dv\leq
\frac{4\left(  m+1\right)  \left(  m-2\right)  }{m^{3}}\int_{M}\left\vert
Ric_{f}^{m}+\frac{trRic_{f}^{m}}{m-2}g\right\vert ^{2}e^{-\frac{m+4}{m}f}dv.
\end{array}
\]
Moreover equality holds if and only if
\[%
\begin{array}
[c]{l}%
Ric_{f}^{m}+\frac{trRic_{f}^{m}}{m-2}g=0.
\end{array}
\]

\bigskip{}

Let $R$ be the nontrivial scalar curvature on $M$ with respect to metric $g,$
and we denote $R_{f}=R+\Delta f,$ $V_{f}(M)=\int_{M}dv_{f}$, $\overline{R_{f}%
}=\frac{\int_{M}R_{f}dv_{f}}{V_{f}(M)},$ $\overline{R}=\frac{\int_{M}Rdv_{f}%
}{V_{f}(M)}$, $R\mathring{i}c=Ric-\frac{R}{n}g$, and $R\mathring{i}%
c_{f}=Ric_{f}-\frac{R_{f}}{n}g.$

Now we state our results:

\begin{theorem}
\label{T2}Let $\left(  M^{n},g,dv_{f}\right)  ,$ $n>2,$ be a closed smooth
metric measure space. If
\[%
\begin{array}
[c]{c}%
Ric_{f}\geq\left(  \Delta f-\left(  n-1\right)  K\right)  g,
\end{array}
\]
then
\begin{equation}%
\begin{array}
[c]{c}%
\left\Vert R_{f}-\overline{R_{f}}\right\Vert _{L^{2}}\leq\frac{2n\sqrt{A}%
}{n-2}\left\Vert R\mathring{i}c_{f}-Hessf\right\Vert _{L^{2}}+\left\Vert
\Delta f\right\Vert _{L^{2}},
\end{array}
\label{eq}%
\end{equation}
where $\left\Vert \cdot\right\Vert _{L^{2}}^{2}=\int_{M}\left(  \cdot\right)
^{2}dv_{f}$,
\[%
\begin{array}
[c]{l}%
A=\frac{n-1}{n}+\frac{1}{\lambda_{1}}\left(  n-1\right)  K,
\end{array}
\]
and $\lambda_{1}$ is the first positive eigenvalue of the weighted Laplacian
$\Delta_{f}.$ Moreover, equality holds if and only if $M$ is trivial Einstein
and constant scalar curvature with respect to metric $g.$
\end{theorem}

\bigskip{}

\begin{theorem}
\label{T1}Let $\left(  M^{n},g,dv_{f}\right)  ,$ $n>2,$ be a closed smooth
metric measure space. If
\[%
\begin{array}
[c]{c}%
Ric_{f}^{m}\geq\left(  \frac{1}{m}\left\vert \nabla f\right\vert ^{2}-\left(
n-1\right)  K\right)  g
\end{array}
\]
for any positive constant $m,$ then
\begin{equation}%
\begin{array}
[c]{c}%
\int_{M}\left(  R-\overline{R}\right)  ^{2}dv_{f}\leq\frac{4n^{2}A}{\left(
n-2\right)  ^{2}}\int_{M}\left\vert R\mathring{i}c\right\vert ^{2}dv_{f}.
\end{array}
\label{eq1}%
\end{equation}
where
\[%
\begin{array}
[c]{c}%
A=\frac{n-1}{n}+\frac{m}{2}+\frac{\left(  m+2\right)  \left(  n-1\right)
}{2\lambda_{1}}K.
\end{array}
\]
Moreover, equality holds if and only if $M$ is trivial Einstein and constant
scalar curvature with respect to metric $g$.
\end{theorem}

\begin{theorem}
\label{T3}Let $\left(  M^{n},g,dv_{f}\right)  ,$ $n>2,$ be a closed smooth
metric measure space. If
\[%
\begin{array}
[c]{c}%
Ric_{f}\geq\left(  \Delta f-\left(  n-1\right)  K\right)  g,
\end{array}
\]
then
\begin{equation}%
\begin{array}
[c]{c}%
\int_{M}\left(  R-\overline{R}\right)  ^{2}dv_{f}\leq\frac{4n^{2}A}{\left(
n-2\right)  ^{2}}\int_{M}\left\vert R\mathring{i}c\right\vert ^{2}dv_{f}%
\end{array}
\label{T3e}%
\end{equation}
where
\[%
\begin{array}
[c]{l}%
A=\frac{n-1}{n}+\frac{1}{\lambda_{1}}\left(  n-1\right)  K.
\end{array}
\]

\end{theorem}

\bigskip{}

\begin{remark}
We note that our results are sharp since the constant $\frac{4n^{2}A}{\left(
n-2\right)  ^{2}}$ is the same as the results in \cite{C1} (for $K>0$) and
\cite{LT} (for $K=0$). So Theorem \ref{T2} implies the results of \cite{LT}
and \cite{C1} whenever we select $f$ is a constant.
\end{remark}

\bigskip{}

\begin{remark}
In Theorem \ref{T2}, \ref{T1}, \ref{T3}, we may select $f$ such that $\int
fdv_{f}=0$ since (\ref{eq}), (\ref{eq1}) and (\ref{T3e}) are valid whenever we
replace $dv_{f}$ by $dv_{f-\overline{f}}.$
\end{remark}

\begin{remark}
In Theorem \ref{T3}, if equality of (\ref{T3e}) holds, the question "$M$ is
trivial Einstein and constant scalar curvature" is still an open problem for
us (see detial in section 3). There are some partial results in the section 3,
and we still have no idea for other case.
\end{remark}

The paper is organized as follows. The proofs of Theorem \ref{T2} and \ref{T1}
are showed in section 2. In section 3, we prove Theorem \ref{T3}, and we also
give some parial results for the open problem in this section.

\section{Proof of Theorem \ref{T2} and \ref{T1}}

\textbf{Proof of Theorem \ref{T2}:}

\proof \ Assume $R$ is the nontrivial scalar curvature on $M$ with respect to
metric $g,$ and $R_{f}=R+\Delta f.$

According to Sobolev embedding theorem and calculus variation, there exists a
nontrivial solution $u:M\rightarrow%
\mathbb{R}%
$ of
\begin{equation}
\left\{
\begin{array}
[c]{c}%
\Delta_{f}u=R_{f}-\overline{R_{f}}\\
\int_{M}udv_{f}=0,
\end{array}
\right.  \label{1R1}%
\end{equation}
where
\[%
\begin{array}
[c]{c}%
\overline{R_{f}}=\frac{\int_{M}R_{f}dv_{f}}{V_{f}(M)}%
\end{array}
\text{ }%
\]
with $V_{f}(M)=\int_{M}dv_{f}$.

Since the second Bianchi identity $\operatorname{div}Ric=\frac{1}{2}\nabla R$ implies%

\[%
\begin{array}
[c]{lll}%
\left(  \operatorname{div}Ric_{f}\right)  _{j} & = & \left(
\operatorname{div}Ric\right)  _{j}+\left(  \operatorname{div}Hessf\right)
_{j}\\
& = & \nabla_{i}R_{ij}+\left(  \operatorname{div}Hessf\right)  _{j}\\
& = & \frac{1}{2}R_{,j}+\left(  \operatorname{div}Hessf\right)  _{j}\\
& = & \frac{1}{2}R_{f,j}-\frac{1}{2}\left(  \Delta f\right)  _{,j}+\left(
\operatorname{div}Hessf\right)  _{j}.
\end{array}
\]
Hence
\[%
\begin{array}
[c]{lll}%
\left(  \operatorname{div}R\mathring{i}c_{f}\right)  _{j} & = & \left(
\operatorname{div}Ric_{f}\right)  _{j}-\frac{R_{f,j}}{n}\\
& = & \frac{n-2}{2n}R_{f,j}-\frac{1}{2}\left(  \Delta f\right)  _{,j}+\left(
\operatorname{div}Hessf\right)  _{j},
\end{array}
\]
i.e.
\begin{equation}%
\begin{array}
[c]{l}%
\operatorname{div}R\mathring{i}c_{f}=\frac{n-2}{2n}\nabla R_{f}-\frac{1}%
{2}\nabla\Delta f+\operatorname{div}Hessf,
\end{array}
\label{001}%
\end{equation}
where $R\mathring{i}c_{f}=Ric_{f}-\frac{R_{f}}{n}g.$ Then, by using
\[%
\begin{array}
[c]{lll}%
\int_{M}\left\langle R\mathring{i}c_{f},hg\right\rangle dv_{f} & = & \int
_{M}\left\langle Ric_{f}-\frac{R_{f}}{n}g,hg\right\rangle dv_{f}\\
& = & \int_{M}\left(  R_{f}-R_{f}\right)  hdv_{f}\\
& = & 0,
\end{array}
\]
one has
\begin{equation}%
\begin{array}
[c]{lll}%
\int_{M}\left(  R_{f}-\overline{R_{f}}\right)  ^{2}dv_{f} & = & \int
_{M}\left(  R_{f}-\overline{R_{f}}\right)  \Delta_{f}udv_{f}=-\int
_{M}\left\langle \nabla R_{f},\nabla u\right\rangle dv_{f}\\
& = & \frac{-2n}{n-2}\int_{M}\left\langle \operatorname{div}R\mathring{i}%
c_{f}+\frac{1}{2}\nabla\Delta f-\operatorname{div}Hessf,\nabla u\right\rangle
dv_{f}\\
& = & \frac{2n}{n-2}\int_{M}\left\langle R\mathring{i}c_{f}-Hessf,Hess_{f}%
u\right\rangle +\frac{1}{2}\Delta f\Delta_{f}udv_{f}\\
& = & \frac{2n}{n-2}\int_{M}\left\langle R\mathring{i}c_{f}-Hessf,Hess_{f}%
u-hg\right\rangle +\frac{n-2}{2n}\Delta f\Delta_{f}udv_{f}\\
& \leq & \frac{2n}{n-2}\left\Vert R\mathring{i}c_{f}-Hessf\right\Vert _{L^{2}%
}\left\Vert Hess_{f}u-hg\right\Vert _{L^{2}}+\int_{M}\Delta f\Delta_{f}udv_{f}%
\end{array}
\label{1a}%
\end{equation}
where $\left\Vert \cdot\right\Vert _{L^{2}}^{2}=\int_{M}\left(  \cdot\right)
^{2}dv_{f},$ and we select
\begin{equation}%
\begin{array}
[c]{l}%
h=\frac{\Delta_{f}u}{n}.
\end{array}
\label{hh}%
\end{equation}

By Bochner's formula
\[%
\begin{array}
[c]{l}%
\frac{1}{2}\Delta_{f}\left\vert \nabla u\right\vert ^{2}=\left\vert
Hessu\right\vert ^{2}+\left\langle \nabla u,\nabla\Delta_{f}u\right\rangle
+Ric_{f}\left(  \nabla u,\nabla u\right)  ,
\end{array}
\]
we have
\begin{equation}%
\begin{array}
[c]{ll}
& \int_{M}\left\vert Hess_{f}u-\frac{\Delta_{f}u}{n}g\right\vert ^{2}dv_{f}\\
= & \int_{M}\left\vert Hess_{f}u\right\vert ^{2}-\frac{\left(  \Delta
_{f}u\right)  ^{2}}{n}dv_{f}\\
= & \int_{M}\left\vert Hessu\right\vert ^{2}-2Hessu\left(  \nabla f,\nabla
u\right)  +\frac{\left\vert \nabla f\right\vert ^{2}\left\vert \nabla
u\right\vert ^{2}+\left\langle \nabla f,\nabla u\right\rangle ^{2}}{2}%
-\frac{\left(  \Delta_{f}u\right)  ^{2}}{n}dv_{f}\\
\leq & \int_{M}\left(  1-\frac{1}{n}\right)  \left(  \Delta_{f}u\right)
^{2}-Ric_{f}\left(  \nabla u,\nabla u\right)  -\left\langle \nabla
f,\nabla\left\vert \nabla u\right\vert ^{2}\right\rangle +\left\vert \nabla
f\right\vert ^{2}\left\vert \nabla u\right\vert ^{2}dv_{f}\\
= & \int_{M}\left(  1-\frac{1}{n}\right)  \left(  \Delta_{f}u\right)
^{2}-Ric_{f}\left(  \nabla u,\nabla u\right)  +\Delta f\left\vert \nabla
u\right\vert ^{2}dv_{f}\\
\leq & \int_{M}\left(  1-\frac{1}{n}\right)  \left(  \Delta_{f}u\right)
^{2}+\left(  n-1\right)  K\left\vert \nabla u\right\vert ^{2}dv_{f}%
\end{array}
\label{1e4}%
\end{equation}
whenever $Ric_{f}\geq\left(  \Delta f-\left(  n-1\right)  K\right)  g.$

Since the first positive eigenvalue $\lambda_{1}$ of weighted Laplacian on $M
$ is characterized by
\[%
\begin{array}
[c]{c}%
\lambda_{1}=\inf\left\{  \frac{\int_{M}\left\vert \nabla\varphi\right\vert
^{2}dv_{f}}{\int_{M}\varphi^{2}dv_{f}}\text{ }|\text{ }\varphi\text{
\textrm{is nontrivial and} }\int_{M}\varphi dv_{f}=0\right\}  ,
\end{array}
\]
we get
\[%
\begin{array}
[c]{lll}%
\int_{M}\left\vert \nabla u\right\vert ^{2}dv_{f} & = & -\int_{M}u\Delta
_{f}udv_{f}=-\int_{M}u\left(  R_{f}-\overline{R_{f}}\right)  dv_{f}%
\leq\left\Vert u\right\Vert _{L^{2}}\left\Vert R_{f}-\overline{R_{f}%
}\right\Vert _{L^{2}}\\
& \leq & \lambda_{1}^{-1/2}||\nabla u||_{L^{2}}\left\Vert R_{f}-\overline
{R_{f}}\right\Vert _{L^{2}},
\end{array}
\]
for which it holds that
\begin{equation}%
\begin{array}
[c]{l}%
\lambda_{1}\int_{M}\left\vert \nabla u\right\vert ^{2}dv_{f}\leq\left\Vert
R_{f}-\overline{R_{f}}\right\Vert _{L^{2}}^{2}\text{ \ \textrm{and } }%
\lambda_{1}^{2}\int_{M}u^{2}dv_{f}\leq\left\Vert R_{f}-\overline{R_{f}%
}\right\Vert _{L^{2}}^{2}.
\end{array}
\label{122}%
\end{equation}
So (\ref{1e4}) becomes
\begin{equation}%
\begin{array}
[c]{c}%
\int_{M}\left\vert Hess_{f}u-\frac{\Delta_{f}u}{n}g\right\vert ^{2}dv_{f}\leq
A\left\Vert R_{f}-\overline{R_{f}}\right\Vert _{L^{2}}^{2},
\end{array}
\text{ } \label{123}%
\end{equation}
where
\[%
\begin{array}
[c]{l}%
A=\frac{n-1}{n}+\frac{1}{\lambda_{1}}\left(  n-1\right)  K,
\end{array}
\]
and we may rewrite (\ref{1a}) as
\[%
\begin{array}
[c]{ll}
& \int_{M}\left(  R_{f}-\overline{R_{f}}\right)  ^{2}dv_{f}\\
\leq & \frac{2n}{n-2}\left\Vert R\mathring{i}c_{f}-Hessf\right\Vert _{L^{2}%
}\left\Vert Hess_{f}u-\frac{\Delta_{f}u}{n}g\right\Vert _{L^{2}}+\int
_{M}\Delta f\Delta_{f}udv_{f}\\
\leq & \frac{2n\sqrt{A}}{n-2}\left\Vert R_{f}-\overline{R_{f}}\right\Vert
_{L^{2}}\left\Vert R\mathring{i}c_{f}-Hessf\right\Vert _{L^{2}}+\int_{M}\Delta
f\left(  R_{f}-\overline{R_{f}}\right)  dv_{f}\\
\leq & \frac{2n\sqrt{A}}{n-2}\left\Vert R_{f}-\overline{R_{f}}\right\Vert
_{L^{2}}\left\Vert R\mathring{i}c_{f}-Hessf\right\Vert _{L^{2}}+\left\Vert
R_{f}-\overline{R_{f}}\right\Vert _{L^{2}}\left\Vert \Delta f\right\Vert
_{L^{2}},
\end{array}
\]
i.e.
\begin{equation}%
\begin{array}
[c]{l}%
\left\Vert R_{f}-\overline{R_{f}}\right\Vert _{L^{2}}\leq\frac{2n\sqrt{A}%
}{n-2}\left\Vert R\mathring{i}c_{f}-Hessf\right\Vert _{L^{2}}+\left\Vert
\Delta f\right\Vert _{L^{2}}.
\end{array}
\label{1e}%
\end{equation}

If $"="$ of (\ref{1e}) holds, then we have the following properties:

\begin{itemize}
\item[(i)] $Ric_{f}\left(  \nabla u,\cdot\right)  =\left(  \Delta f-\left(
n-1\right)  K\right)  g\left(  \nabla u,\cdot\right)  ;$

\item[(ii)] $\mu_{1}(R\mathring{i}c_{f}-Hessf)=Hess_{f}u-\frac{\Delta_{f}u}%
{n}g,$ where $\mu_{1}$ is a non-zero constant;

\item[(iii)] $R_{f}-\overline{R_{f}}=-\lambda_{1}u=\Delta_{f}u=\mu_{2}\Delta
f,$ where $\mu_{2}$ is a non-zero constant;

\item[(iv)] $\nabla f=\alpha\nabla u.$
\end{itemize}

\bigskip{}

Since (iii) and (iv) imply
\[
0=\int_{M}\Delta_{f}udv_{f}=\mu_{2}\int_{M}\Delta fdv_{f}=\mu_{2}\int
_{M}\left\vert \nabla f\right\vert ^{2}dv_{f}%
\]
infers that $f$ must be constant on $M.$ Then theorem follows by \cite{C1} and
\cite{LT}.\ \endproof

\bigskip{}

\textbf{Proof of Theorem \ref{T1} :}

In the following, we show almost Schur lemma under the condition of
$m$-Bakry-Émery Ricci tensor which is similar to the work of \cite{W}. \ 

\bigskip{}

\proof Now we consider the nontrivial solution $u:M\rightarrow%
\mathbb{R}%
$ of
\begin{equation}
\left\{
\begin{array}
[c]{c}%
\Delta_{f}u=R-\overline{R}\\
\int_{M}udv_{f}=0
\end{array}
\right.  \label{a0}%
\end{equation}
where
\[%
\begin{array}
[c]{c}%
\overline{R}=\frac{\int_{M}Rdv_{f}}{V_{f}(M)}.
\end{array}
\text{ }%
\]

Since the second Bianchi identity $\operatorname{div}Ric=\frac{1}{2}\nabla R$
implies
\[%
\begin{array}
[c]{l}%
\operatorname{div}R\mathring{i}c=\frac{n-2}{2n}\nabla R
\end{array}
\]
where $\left(  \operatorname{div}Ric\right)  _{j}=\nabla_{i}R_{ij}$ and
$R\mathring{i}c=Ric-\frac{R}{n}g.$

Then, one has
\begin{equation}%
\begin{array}
[c]{lll}%
\int_{M}\left(  R-\overline{R}\right)  ^{2}dv_{f} & = & \int_{M}\left(
R-\overline{R}\right)  \Delta_{f}udv_{f}=-\int_{M}\left\langle \nabla R,\nabla
u\right\rangle dv_{f}\\
& = & \frac{-2n}{n-2}\int_{M}\left\langle \operatorname{div}R\mathring
{i}c,\nabla u\right\rangle dv_{f}\\
& = & \frac{2n}{n-2}\int_{M}\left\langle R\mathring{i}c,Hess_{f}u\right\rangle
dv_{f}\\
& = & \frac{2n}{n-2}\int_{M}\left\langle R\mathring{i}c,Hess_{f}%
u-hg\right\rangle dv_{f}\\
& \leq & \frac{2n}{n-2}\left\Vert R\mathring{i}c\right\Vert _{L^{2}}\left\Vert
Hess_{f}u-hg\right\Vert _{L^{2}}.
\end{array}
\label{a1}%
\end{equation}

Now we select
\[%
\begin{array}
[c]{l}%
h=\frac{\Delta_{f}u}{n}.
\end{array}
\]

By Bochner's formula
\[%
\begin{array}
[c]{l}%
\frac{1}{2}\Delta_{f}\left\vert \nabla u\right\vert ^{2}=\left\vert
Hessu\right\vert ^{2}+\left\langle \nabla u,\nabla\Delta_{f}u\right\rangle
+Ric_{f}\left(  \nabla u,\nabla u\right)  ,
\end{array}
\]
we have%

\begin{equation}%
\begin{array}
[c]{ll}
& \int_{M}\left\vert Hess_{f}u-\frac{\Delta_{f}u}{n}g\right\vert ^{2}dv_{f}\\
= & \int_{M}\left\vert Hessu-\nabla f\otimes\nabla u\right\vert ^{2}%
-\frac{\left(  \Delta_{f}u\right)  ^{2}}{n}dv_{f}\\
\leq & \int_{M}\left(  1+\frac{m}{2}\right)  \left\vert Hessu\right\vert
^{2}+\left(  1+\frac{2}{m}\right)  \left\vert \nabla f\otimes\nabla
u\right\vert ^{2}-\frac{\left(  \Delta_{f}u\right)  ^{2}}{n}dv_{f}\\
= & \int_{M}\left(  1-\frac{1}{n}+\frac{m}{2}\right)  \left(  \Delta
_{f}u\right)  ^{2}-\frac{m+2}{2}Ric_{f}\left(  \nabla u,\nabla u\right) \\
& +\frac{m+2}{2m}\left(  \left\vert \nabla f\right\vert ^{2}\left\vert \nabla
u\right\vert ^{2}+\left\langle \nabla f,\nabla u\right\rangle ^{2}\right)
dv_{f}\\
\leq & \int_{M}\left(  \frac{n-1}{n}+\frac{m}{2}\right)  \left(  \Delta
_{f}u\right)  ^{2}+\frac{\left(  m+2\right)  \left(  n-1\right)  K}%
{2}\left\vert \nabla u\right\vert ^{2}dv_{f},
\end{array}
\label{a2}%
\end{equation}
here we use $Ric_{f}^{m}\geq\left(  \frac{1}{m}\left\vert \nabla f\right\vert
^{2}-\left(  n-1\right)  K\right)  g.$

Hence, by the inequality of eigenvalue $\lambda_{1}$ (see (\ref{122})),
(\ref{a2}) gives
\[
\int_{M}\left\vert Hess_{f}u-\frac{\Delta_{f}u}{n}g\right\vert ^{2}dv_{f}\leq
A\left\Vert R-\overline{R}\right\Vert _{L^{2}}^{2},
\]
and then one has
\begin{equation}%
\begin{array}
[c]{l}%
\left\Vert R-\overline{R}\right\Vert _{L^{2}}\leq\frac{2n\sqrt{A}}%
{n-2}\left\Vert R\mathring{i}c\right\Vert _{L^{2}}%
\end{array}
\label{a10}%
\end{equation}
where
\[%
\begin{array}
[c]{c}%
A=\left(  \frac{n-1}{n}+\frac{m}{2}\right)  +\frac{(m+2)(n-1)}{2\lambda_{1}}K
\end{array}
\]
for any constant $m>0.$

If $"="$ holds, then $Hessu=-\frac{2}{m}\nabla f\otimes\nabla u$ on $M.$ For
which it implies
\begin{equation}%
\begin{array}
[c]{c}%
\Delta_{\frac{-2f}{m}}u=\Delta u+\frac{2}{m}\left\langle \nabla f,\nabla
u\right\rangle =0,
\end{array}
\label{e1}%
\end{equation}
i.e. $u$ is a weighted harmonic function with respect to weighted measure
$dv_{\frac{2f}{m}}$ on $M,$ it infers $u=0$ on $M$. So Theorem \ref{T1}
follows. \ \endproof

\bigskip

Combine Theorem \ref{T1} and Theorem \ref{T2}, it is clear that one has the
following property.

\bigskip

\begin{corollary}
Let $\left(  M^{n},g,dv_{f}\right)  ,$ $n>2,$ be a closed smooth metric
measure space. If
\[%
\begin{array}
[c]{c}%
Ric_{f}^{m}\geq\left(  \frac{1}{m}\left\vert \nabla f\right\vert ^{2}-\left(
n-1\right)  K\right)  Kg
\end{array}
\]
for any positive constant $m,$ then%
\[%
\begin{array}
[c]{c}%
\left\Vert R_{f}-\overline{R_{f}}\right\Vert _{L^{2}}\leq\frac{2n\sqrt{A}%
}{n-2}\left\Vert R\mathring{i}c_{f}-Hessf\right\Vert _{L^{2}}+\left\Vert
\Delta f\right\Vert _{L^{2}},
\end{array}
\]
where $\left\Vert \cdot\right\Vert _{L^{2}}^{2}=\int_{M}\left(  \cdot\right)
^{2}dv_{f}$ and
\[%
\begin{array}
[c]{c}%
A=\frac{n-1}{n}+\frac{m}{2}+\frac{\left(  m+2\right)  \left(  n-1\right)
}{2\lambda_{1}}K.
\end{array}
\]
Moreover, equality holds if and only if $M$ is trivial Einstein and constant
scalar curvature with respect to metric $g$.
\end{corollary}

\section{Proof of Theorem \ref{T3}}

As the procedure form (\ref{a0}) to (\ref{a10}) in the proof of Theorem
\ref{T1}, but we replace (\ref{a2}) by the following formula,
\[%
\begin{array}
[c]{ll}
& \int_{M}\left\vert Hess_{f}u-\frac{\Delta_{f}u}{n}g\right\vert ^{2}dv_{f}\\
= & \int_{M}\left\vert Hess_{f}u\right\vert ^{2}-\frac{\left(  \Delta
_{f}u\right)  ^{2}}{n}dv_{f}\\
= & \int_{M}\left\vert Hessu\right\vert ^{2}-2Hessu\left(  \nabla f,\nabla
u\right)  +\frac{\left\vert \nabla f\right\vert ^{2}\left\vert \nabla
u\right\vert ^{2}+\left\langle \nabla f,\nabla u\right\rangle ^{2}}{2}%
-\frac{\left(  \Delta_{f}u\right)  ^{2}}{n}dv_{f}\\
\leq & \int_{M}\left(  1-\frac{1}{n}\right)  \left(  \Delta_{f}u\right)
^{2}-Ric_{f}\left(  \nabla u,\nabla u\right)  -\left\langle \nabla
f,\nabla\left\vert \nabla u\right\vert ^{2}\right\rangle +\left\vert \nabla
f\right\vert ^{2}\left\vert \nabla u\right\vert ^{2}dv_{f}\\
= & \int_{M}\left(  1-\frac{1}{n}\right)  \left(  \Delta_{f}u\right)
^{2}-Ric_{f}\left(  \nabla u,\nabla u\right)  +\Delta f\left\vert \nabla
u\right\vert ^{2}dv_{f}\\
\leq & \int_{M}\left(  1-\frac{1}{n}\right)  \left(  \Delta_{f}u\right)
^{2}+\left(  n-1\right)  K\left\vert \nabla u\right\vert ^{2}dv_{f}%
\end{array}
\]
whenever $Ric_{f}\geq\left(  \Delta f-\left(  n-1\right)  K\right)  g.$ So we
obtain
\begin{equation}%
\begin{array}
[c]{c}%
\int_{M}\left\vert Hess_{f}u-\frac{\Delta_{f}u}{n}g\right\vert ^{2}dv_{f}\leq
A\int_{M}\left(  R-\overline{R}\right)  ^{2}dv_{f},
\end{array}
\label{w1}%
\end{equation}
and then the inequality (\ref{T3e})
\begin{equation}%
\begin{array}
[c]{c}%
\int_{M}\left(  R-\overline{R}\right)  ^{2}dv_{f}\leq\frac{4n^{2}A}{\left(
n-2\right)  ^{2}}\int_{M}\left\vert R\mathring{i}c\right\vert ^{2}dv_{f}%
\end{array}
\label{w2}%
\end{equation}
holds, where
\[%
\begin{array}
[c]{l}%
A=\frac{n-1}{n}+\frac{1}{\lambda_{1}}\left(  n-1\right)  K.
\end{array}
\]

\bigskip{}

If $"="$ of (\ref{w2}) holds, we have the following properties:

\begin{itemize}
\item[(i)] $Ric_{f}\left(  \nabla u,\cdot\right)  =\left(  \Delta f-\left(
n-1\right)  K\right)  g\left(  \nabla u,\cdot\right)  ;$

\item[(ii)] $\mu R\mathring{i}c=Hess_{f}u-\frac{\Delta_{f}u}{n}g,$ where $\mu$
is a non-zero constant;

\item[(iii)] $R-\overline{R}=-\lambda_{1}u;$

\item[(iv)] $\nabla f=\alpha\nabla u.$
\end{itemize}

\bigskip{}

In the following, we focus on the case $K>0,$ we prove that if $"="$ of
(\ref{w2}) holds, and under the condition $\alpha\left(  p\right)  \leq
\frac{1}{n-1},$ then $M$ is trivial Einstein and constant scalar curvature
with respect to metric $g$, but it is still an open problem when
$\alpha\left(  p\right)  >\frac{1}{n-1}.$ In fact, by the compactness property
of $M,$ we only need to consider the case $\frac{1}{n-1}<\alpha\left(
p\right)  \leq C,$ for some constant $C.$

\begin{remark}
If $\alpha=0,$ then theorem follows by \cite{C1} (or \cite{LT} for $K=0$). So
we assume $\alpha\neq0.$
\end{remark}

\begin{lemma}
\label{ke1}$\mu$ must satisfy $\mu=\frac{2nA}{n-2}=\frac{2n}{n-2}\left(
\frac{n-1}{n}+\frac{1}{\lambda_{1}}\left(  n-1\right)  K\right)  ,$ or $M$ is
trivial Einstein and constant scalar curvature with respect to metric $g.$
\end{lemma}

\begin{proof}
By (ii), (\ref{w1}) and (\ref{w2}), we have
\[%
\begin{array}
[c]{lllll}%
\mu^{2}\int_{M}\left\vert R\mathring{i}c\right\vert ^{2}dv_{f} & = & \int
_{M}\left\vert Hess_{f}u-\frac{\Delta_{f}u}{n}g\right\vert ^{2}dv_{f} & = &
A\left\Vert R-\overline{R}\right\Vert _{L^{2}}^{2}\\
& = & \frac{4n^{2}A^{2}}{\left(  n-2\right)  ^{2}}\int_{M}\left\vert
R\mathring{i}c\right\vert ^{2}dv_{f} &  &
\end{array}
\]
which gives
\[%
\begin{array}
[c]{c}%
\left(  \mu^{2}-\frac{4n^{2}A^{2}}{\left(  n-2\right)  ^{2}}\right)  \int
_{M}\left\vert R\mathring{i}c\right\vert ^{2}dv_{f}=0.
\end{array}
\]
It is clear that if $\mu\neq\frac{2nA}{n-2},$ then $R\mathring{i}c=0$ and $M$
is trivial Einstein and constant scalar curvature with respect to metric $g$.
Hence $\mu=\frac{2nA}{n-2}.$
\end{proof}

\bigskip{}

By (i),
\[%
\begin{array}
[c]{l}%
R_{ij}u_{i}+f_{ij}u_{i}-u_{j}f_{ii}+\left(  n-1\right)  Ku_{j}=0
\end{array}
\]
implies
\begin{equation}%
\begin{array}
[c]{c}%
R_{ij,j}u_{i}+f_{ijj}u_{i}-f_{iij}u_{j}+R_{ij}u_{ij}+f_{ij}u_{ij}-u_{jj}%
f_{ii}+\left(  n-1\right)  Ku_{jj}=0.
\end{array}
\label{3e5}%
\end{equation}
And (ii) gives
\begin{equation}%
\begin{array}
[c]{lll}%
\mu R_{ij}u_{,ij} & = & \left\langle \frac{\mu R}{n}g+Hess_{f}u-\frac
{\Delta_{f}u}{n}g,Hessu\right\rangle \\
& = & \frac{\mu R}{n}\Delta u+\left\vert Hessu\right\vert ^{2}-\alpha
Hessu\left(  \nabla u,\nabla u\right)  -\frac{\Delta u-\alpha\left\vert \nabla
u\right\vert ^{2}}{n}\Delta u.
\end{array}
\label{3e6}%
\end{equation}

Now let $p\in M$ be the minimal point $p$ of $u,$ i.e. $u\left(  p\right)
=\inf_{M}u.$

Since (iv) implies%
\[%
\begin{array}
[c]{l}%
f_{ik}=\alpha_{k}u_{i}+\alpha u_{ik},
\end{array}
\]
then (\ref{3e5}) and (\ref{3e6}) can be rewritten as
\begin{equation}
\left\{
\begin{array}
[c]{lll}%
R_{ij}u_{,ij} & = & \alpha\left(  \Delta u\right)  ^{2}-\alpha\left\vert
Hessu\right\vert ^{2}-\left(  n-1\right)  K\Delta u,\\
\mu R_{ij}u_{,ij} & = & \frac{\mu R}{n}\Delta u+\left\vert Hessu\right\vert
^{2}-\frac{1}{n}\left(  \Delta u\right)  ^{2},
\end{array}
\right.  \label{3e7}%
\end{equation}
at $p.$ For which we have
\begin{equation}%
\begin{array}
[c]{lll}%
0 & = & \frac{\mu R}{n}\Delta u+\left(  1+\alpha\mu\right)  \left\vert
Hessu\right\vert ^{2}-\left(  \frac{1}{n}+\alpha\mu\right)  \left(  \Delta
u\right)  ^{2}+\left(  n-1\right)  \mu K\Delta u\\
& = & \left(  1+\alpha\mu\right)  \left\vert Hessu\right\vert ^{2}%
-\frac{1+\alpha\mu}{n}\left(  \Delta u\right)  ^{2}+\frac{\mu}{n}\left(
R-\left(  n-1\right)  \alpha\Delta u+n\left(  n-1\right)  K\right)  \Delta u
\end{array}
\label{1w0}%
\end{equation}
at $p.$

Since
\[%
\begin{array}
[c]{lll}%
R-\left(  n-1\right)  \alpha\Delta u+n\left(  n-1\right)  K & = & \overline
{R}+\Delta u-\left(  n-1\right)  \alpha\Delta u+n\left(  n-1\right)  K\\
& = & \left(  1-\left(  n-1\right)  \alpha\right)  \Delta u+\overline
{R}+n\left(  n-1\right)  K,
\end{array}
\]
so (\ref{1w0}) can be rewritten as
\begin{equation}%
\begin{array}
[c]{ll}
& \left(  1+\alpha\mu\right)  \left(  \left\vert Hessu\right\vert ^{2}%
-\frac{1}{n}\left(  \Delta u\right)  ^{2}\right)  +\frac{\mu}{n}\left(
1-\left(  n-1\right)  \alpha\right)  \left(  \Delta u\right)  ^{2}\\
& +\frac{\mu}{n}\left(  \overline{R}+n\left(  n-1\right)  K\right)  \Delta u\\
= & 0\text{ at }p.
\end{array}
\label{1w2}%
\end{equation}

Besides, due to curvature assumption
\begin{equation}%
\begin{array}
[c]{l}%
Ric+\alpha Hessu\geq\left(  \alpha\Delta u-\left(  n-1\right)  K\right)  g,
\end{array}
\label{1w10}%
\end{equation}
one has
\begin{equation}%
\begin{array}
[c]{l}%
R\geq\alpha\left(  n-1\right)  \Delta u-n\left(  n-1\right)  K,
\end{array}
\label{1w11}%
\end{equation}
and it gives%

\begin{equation}%
\begin{array}
[c]{lll}%
\overline{R} & \geq & \frac{\alpha\left(  n-1\right)  }{V_{f}\left(  M\right)
}\int_{M}\Delta udv_{f}-n\left(  n-1\right)  K\\
& = & \frac{\alpha^{2}\left(  n-1\right)  }{V_{f}\left(  M\right)  }\int
_{M}\left\vert \nabla u\right\vert ^{2}dv_{f}-n\left(  n-1\right)  K\\
& > & -n\left(  n-1\right)  K\text{ for all }\alpha\neq0,
\end{array}
\label{1w1}%
\end{equation}
here we use
\[%
\begin{array}
[c]{c}%
\Delta u=\Delta_{f}u+\alpha\left\vert \nabla u\right\vert ^{2}.
\end{array}
\]

If $\frac{-1}{\mu}\leq\alpha\left(  p\right)  \leq\frac{1}{n-1},$ then each
term in the left hand side of (\ref{1w2}) must be nonnegative at $p,$ so
$\Delta u\left(  p\right)  =0,$ which implies $R\left(  p\right)  =\sup_{M}R=$
$\overline{R}.$ Hence $M$ is trivial Einstein and constant scalar curvature
with respect to metric $g.$

If $\alpha\left(  p\right)  \leq-\frac{1}{\mu},$ at $p,$ we rewrite
(\ref{1w2}) as
\begin{equation}%
\begin{array}
[c]{l}%
\left(  1+\alpha\mu\right)  \left(  \left\vert Hessu\right\vert ^{2}-\left(
\Delta u\right)  ^{2}\right)  +\frac{\left(  n-1\right)  +\mu}{n}\left(
\Delta u\right)  ^{2}+\frac{\mu}{n}\left(  \overline{R}+n\left(  n-1\right)
K\right)  \Delta u=0,
\end{array}
\label{1w02}%
\end{equation}
at $p.$ We note that, at $p$, the $n\times n$ matrix $Hessu$ must be
semi-positive, then $\left\vert Hessu\right\vert ^{2}\leq\left(  \Delta
u\right)  ^{2}$ at $p$, and equality holds if only if the rank of
$Hessu\left(  p\right)  $ less than $2.$ For which each term in the left hand
side of (\ref{1w02}) must be nonnegative. So $\Delta u\left(  p\right)
=R\left(  p\right)  -\overline{R}=0,$ then $M$ is trivial Einstein and
constant scalar curvature with respect to metric $g$.

\bigskip{}

\bigskip{}

\bigskip{}

\begin{thebibliography}{99}                                                                                               %


\bibitem {AN}B. Andrews and L. Ni, \textit{Eigenvalue comparison on
Bakry-Emery manifolds}, Comm. Partial Differential Equations 37 (2012), no.
11, 2081--2092.

\bibitem {B}E.R. Barbosa, \textit{A note on the almost-Schur lemma on
4-dimensional Riemannian closed manifolds}, Proc. Amer. Math. Soc. 140 (2012), 4319-4322.

\bibitem {BQ}D. Bakry and Z. Qian, \textit{Some new results on eigenvectors
via dimension, diameter, and Ricci flow}, Adv. Math. 155 (2000), 98--153.

\bibitem {CSW}J.-T. Chen, T. Saotome and C.-T. Wu, \textit{The CR Almost Schur
lemma and Lee conjecture}, Kyoto J. Math., 52 (2012), no. 1, 89-98.

\bibitem {CDW}J.-T. Chen, N. T. Dung and C.-T. Wu, \textit{A generalization of
almost Schur lemma on CR manifolds}, arXiv:1405.3038 {[}math.DG{]}

\bibitem {C1}X. Cheng, \textit{A generalization of almost Schur lemma for
closed Riemannian manifolds}, Ann. Global Anal. Geom. 43 (2013), no. 2, 153--160.

\bibitem {C2}X. Cheng, \textit{An almost-Schur lemma for symmetric (2;0)
tensors and applications}, Pacific J. Math. 267 (2014) 325-340

\bibitem {CZ}Xu Cheng, Detang Zhou, Eigenvalues of the drifted Laplacian on
complete metric measure spaces, Commun. Contemp. Math. 19, no. 1 (2017),
1650001, 17 pp.

\bibitem {GW}Y. Ge and G. Wang, \textit{An almost Schur Theorem on
4-dimensional manifolds}, Proc. Amer. Math. Soc. 140 (2012), 1041--1044.

\bibitem {H}P. T. Ho, \textit{Almost Schur lemma for manifolds with boundary}.
(English summary)Differential Geom. Appl. 32 (2014), 97--112. MR3147198

\bibitem {LT}De Lellis and C. Topping, \textit{Almost Schur lemma}, Calc. Var.
and PDE, 43 (2012) 347-354.

\bibitem {LY}P. Li and S.T. Yau, \textit{Eigenvalues of a compact Riemannian
manifold}. AMS Proc. Symp. Pure Math. 36 (1980), 205--239

\bibitem {MW1}O. Munteanu and J. Wang, \textit{Smooth metric measure spaces
with non-negative curvature}. Comm. Anal. Geom. 19 (2011), no. 3, 451--486.

\bibitem {MW2}O. Munteanu and J. Wang, \textit{Analysis of weighted Laplacian
and applications to Ricci solitons}. Comm. Anal. Geom. 20 (2012), no. 1, 55--94.

\bibitem {W}J.-Y. Wu, \textit{De Lellis-Topping type inequalities for smooth
metric measure spaces}, Geom. Dedicata 169 (2014), 273--281. 53C21 (53C24)

\bibitem {WE}G. Wei, \textit{Volume comparison and its generalizations},
Advanced Lecture in Mathematics, Vol. 22 (2012), 311-322.

\bibitem {WCS}G. Wei and J. Case and Y. Shu, \textit{Rigidity of
Quasi-Einstein Metrics}, Diff. Geom. and its Applications 29 (2011), 93-100.

\bibitem {WW}G. Wei and W. Wylie, \textit{Comparison Geometry for the
Bakry-Emery Ricci Tensor}, J. Diff. Geom. 83, no. 2 (2009), 377-405.
\end{thebibliography}
\end{document}